\documentclass{article}
\usepackage{graphicx}
\usepackage{color}
\newcommand{\cR}{\mathcal{R}}
\newcommand{\R}{\mathbf{R}}
\newcommand{\N}{\mathbf{N}}
\usepackage{amsmath,amsthm}

\newtheorem{lemma}{Lemma}

\title{A note on numerical singular values of compositions with non-compact operators}

\author{Daniel Gerth\footnote{Technische Universit\"at Chemnitz,
Fakult\"at f\"ur Mathematik, D-09107 Chemnitz, Germany, ({\tt daniel.gerth@mathematik.tu-chemnitz.de})}}

\begin{document}

\maketitle

\begin{abstract}
Linear non-compact operators are difficult to study because they do not exist in the finite dimensional world. Recently, Math\'{e} and Hofmann studied the singular values of the compact composition of the non-compact Hausdorff moment operator and the compact integral operator and found credible arguments, but no strict proof, that those singular values fall only slightly faster than those of the integral operator alone. However, the fact that numerically the singular values of the combined operator fall exponentially fast was not mentioned. In this note, we provide the missing numerical results and provide an explanation why the two seemingly contradicting results may both be true.
\end{abstract}



\section{Introduction} 
For numerical computations one often needs to find a finite dimensional approximation to a real world problem that can be modelled as an infinite dimensional operator equation. A curious case is when a non-compact operator is involved in the modelling. The reason for this is that any linear operator with finite dimensional range (in particular discrete operators used for computations) is necessarily compact. On the other hand, whenever a non-compact operator is paired with a compact operator, the composition is compact again. It is natural to ask whether the non-compact part can ``destroy" important properties of the compact one. We will investigate this in terms of the singular values.

The background of this note is the approximate solution of inverse problems. Let $X$ and in particular $Y$ be infinite dimensional Hilbert spaces, and $A:X\rightarrow Y$ a bounded linear operator. An inverse problem, in the most basic definition, is to recover the cause $x^\dag\in X$ from data $y=Ax^\dag$, where often only a noisy approximation $y^\delta$ to $y$ is available. Such problems become challenging when the operator $A$ is ill-posed, as then $A$ may not be invertible, and the pseudoinverse $A^{-1}$ is discontinuous, meaning arbitrarily small measurement errors can lead to arbitrary errors in the reconstruction of $x^\dag$. A linear operator is ill-posed if and only if its range is not closed, $\cR(A)\neq\overline{\cR(A)}$ \cite{EHN}. Nasheds definition of ill-posedness \cite{Nashed} further distinguishes between ill-posedness of type I, where $\cR(A)$ contains an infinite dimensional closed subspace, and ill-posedness of type II, where this is not the case. In the setting where $X$ and $Y$ are Hilbert-spaces, ill-posedness of type II occurs if and only if $A$ is compact, and the theory for the inversion of compact linear operators has been treated abundantly in the literature; see, e.g., \cite{EHN}. The degree of ill-posedness, which is of high importance for the design of appropriate regularization schemes to find reasonably good approximations to $x^\dag$, is characterized by the decay of the singular values $\sigma_i(A)$ of $A$. Denoting by $A^\ast:Y\rightarrow X$ the adjoint of $A$, $\sigma_i$ is a singular value of $A$ if $\sigma_i^2$ is an eigenvalue of $A^\ast A$. For compact operators, the non-increasing singular values are bounded from above with $\sigma_1(A)=\|A\|$, are countable and can only accumulate in zero. The singular values and, more generally, the singular system $\{\sigma_i,u_i,v_i\}$, where $u_i$ is an ONB for $\cR(A)$ and $v_i$ is an ONB for $\cR(A^\ast)$ is closely related to the singular value decomposition $\hat A=USV^T$ of a matrix-approximation to $A$. 

In the Hilbert-space setting, it follows that ill-posedness of type I corresponds to non-compact operators with non-closed range. It appears that generally inverse problems with such an operator are considered less ill-posed in the literature \cite{Nashed}, but since the operators do not possess a singular system, it is difficult to formulate an analogon to the degree of ill-posedness. However, if
\[
A=B\circ J
\]
is a composition between a non-compact operator $B$ and a compact operator $J$, then $A$ is also compact. This leads us to the main question motivating this note: Is the degree of ill-posedness of $A$ equal or at least similar to the one of $J$, or can the application of the non-compact $B$ yield a (significantly) more ill-posed problem? 

We will only consider concrete choices of the operators. Our main interest for the non-compact operator is the Hausdorff-moment (HM) operator \cite{Hausdorff23}
\begin{equation}\label{eq:hausdorff}
B^H: L^2[0,1]\rightarrow \ell^2,\quad [B^Hx]_i=\int_0^1 t^{i-1} x(t)\,dt.
\end{equation}
Since the compact part is of less interest, we choose for simplicity the integral operator
\begin{equation}\label{eq:int}
J: L^2[0,1]\rightarrow L^2[0,1],\quad [Jx](s)=\int_0^s x(t)\,dt.
\end{equation}
Note that the singular system of $J$ is well-known \cite{EHN}, in particular the singular values fall as $\sigma_i(J)\sim 1/i$. Finally, since the HM operator is sometimes difficult to handle, we will use the multiplication operator 
\begin{equation}\label{eq:mult}
B^M:L^2[0,1]\rightarrow L^2[0,1], \quad [B^Mx](s)=m(s)x(s),
\end{equation}
with multiplicator function $m(s)$ containing essential zeros in $[0,1]$, to draw analogies. 

Recently, the paper \cite{GHHK21} provided some new insight on the HM-operator in relation to inverse problems. Even more, Math\'{e} and Hofmann discussed our main question in \cite{HM} from an analytical point of view. They showed that there are positive constants $\underline c$, $\bar c$ such that
\begin{equation}\label{eq:result_MH}
\exp(-\underline c i)\leq \sigma_i(B^H\circ J)\leq \frac{\bar c}{i}.
\end{equation}
This leaves open whether the singular values fall just like the ones of the integration operator, or if the HM operators makes the composition exponentially ill-posed. Math\'{e} and Hofmann then proceed to show that
\begin{equation}\label{eq:sing_HJ_O}
\sigma_i(B^H\circ J)=\mathcal{O}\left(i^{-\frac{3}{2}}\right),
\end{equation}
i.e., the composite operator is at least slightly more ill-posed than $J$ alone, but exponential decay is not ruled out. They also provide a substantial argument that \eqref{eq:sing_HJ_O} is the actual degree of ill-posedness, and that the singular values of $A=B^H\circ J$ can not fall exponentially fast. The reasoning for this is that the operator $A^\ast A$ can be written as Fredholm integral operator
\begin{equation}\label{eq:fredholm}
[A^\ast Ax](s)=\int_0^1 k(s,t) x(t)\,dt
\end{equation}
with kernel
\begin{equation}\label{eq:kernel}
k(s,t)=\sum_{j=1}^\infty \frac{(1-s^j)(1-t^j)}{j^2}, \qquad 0\leq s,t\leq 1.
\end{equation}
While $k$ is continuous, its derivative $\frac{\partial}{\partial s} k(s,t)$ has a pole at $s=1$. In particular, $k$ is not Lipshitz-continuous. It is well-known that for integral operators as in \eqref{eq:fredholm}, exponentially decaying singular values are usually associated with $k\in C^\infty$, i.e., smooth, infinitely many times differentiable kernels. The fact that the smoothness of \eqref{eq:kernel} is highly limited suggests that the singular values of $A$ do not fall exponentially fast, but there seems to be no theorem available in the literature that covers this situation. Therefore, ultimately, the exponential decay seems inlikely but can not be ruled out.

What is missing in the work of Math\'{e} and Hofmann is a mentioning of the fact that numerically, when $A=B^H \circ J$ is approximated by a (finite) matrix $\hat A$, its singular values do fall exponentially. This is somewhat surprising, in particular in comparison to the composition of $J$ with the multiplication operator. In this case, it is well-known \cite{HW05,HW09} that in particular for the multiplicator functions $m(s)=s^\kappa$, $\kappa>0$, (which we will use throughout the paper), the multiplication operator  does not change the degree of ill-posedness,
\begin{equation}\label{eq:svs_BMJ}
\sigma_i(B^M\circ J)\sim 1/i,
\end{equation}
and numerical observations agree with this. We mention here that the operators $B^H$ and $B^M$ have similar properties. Both are non-compact with non-closed range, and both have a purely continuous spectrum; the interval $[0,\sqrt{\pi}]$ for $B^H$ and all values $m(s)$, $0\leq s\leq 1$, for $B^M$. With our choice of $m(s)$ the latter corresponds to the interval $[0,1]$.

In the remainder of this paper, we make an attempt to explain this seeming mismatch. We show numerical experiments for both operators $A=B^H\circ J$, $A=B^M\circ J$, in Section \ref{sec:numerics}. After that, we present a reasoning that allows all statements on the singular values made so far to be true without a contradiction.

\section{Numerical singular values}\label{sec:numerics}
Since the theoretical approach to the estimation of the singular values of the composite operator $A=B^{(H)}\circ J$ in \cite{HM} was not entirely conclusive we study in this section the computed singular values of the linear systems obtain by discretizing the infinite dimensional problems. We are mainly interested in the singular values of the composition $A=B^H\circ J$ with the Hausdorff operator $B^H$, but we also report on the singular values of the composition $\tilde A=B^M\circ J$ with a multiplication operator. 

To discretize the operators we need to approximate integrals of type 
\begin{equation}\label{eq:int_prototype}
[I_kx](s)=\int_0^1 k(s,t)x(t)\,dt.
\end{equation} 
Namely, we have (with slight abuse of notation for $B^H$)
\begin{align*}
&[Jx](s)= [I_k x](s)\mbox{ with } k(s,t)=\begin{cases} 1 & t\leq s \\ 0 & else\end{cases},\\
&[B^{H}x]_j= [I_kx]_j\mbox{ with } k(j,t)=t^{j-1}.
\end{align*}
It was shown in \cite{HM} that the composition $B^H\circ J$ can be written as
\[
[Ax]_j=[B^{H}\circ J (x)]_j=\int_0^1 \frac{1}{j}(1-t^j)x(t)\,dt
\]
hence
\[
[Ax]_j=[B^{(H)}\circ J]_j=[I_kx]_j\mbox{ with } k(j,t)=\frac{1}{j}(1-t^{j-1}).
\]
We discretize the integrals $I_k$ with the trapezoidal rule, i.e., we use a grid $t_i$, $i=1,\dots,N$, where $N$ is given in the examples, such that $t_i=(i-1)/(N-1)$. Analogously we discretize the variable $s\in[0,1]$ as $s_j$, $j=1,\dots,M$, such that
\begin{equation}\label{eq:int_num}
[I_kx](s_j)=\frac{1}{N-1}\left(\frac{1}{2} k(s_j,t_1)x(t_1)+\sum_{i=2}^{N-1} k(s_j,t_i)x(t_i)+\frac{1}{2} k(s_j,t_N)x(t_N)\right).
\end{equation}
There are integration rules that yield a better rate of approximation of the true integral. However, we consider the trapezoidal rule as a standard implementation, and as we will see it is enough to verify \eqref{eq:svs_BMJ}. For the Hausdorff operator we will use a discretization fine enough to make it unlikely that the choice integration rule alone explains the discrepancies in the singular values. We will refer back to this statement in the next section.

The discretization \eqref{eq:int_num} yields $M\times N$ matrix approximations to the respective operators, which we indicate with a hat symbol, i.e., we have $\hat B^{H}\in \R^{M\times N}$ as approximation to \eqref{eq:hausdorff}, $\hat J\in \R^{M\times N}$ as approximation to \eqref{eq:int}, $\hat A \in \R^{M\times N}$ as approximation to the composition $A=B^{H}\circ J$. Finally we have the multiplication operator $ B^{M}$ from \eqref{eq:mult}, which is approximated through a diagonal matrix $\hat B^{M}$ with entries $\hat B^{M}(i,i)=m(t_i)$, $i=1,\dots,N$.
Note that due to discretization the representations  $\hat B^M$ and $\hat B^{H}$ are compact, whereas the infinite dimensional operators are not. We can therefore calculate the singular values for the discretized operators.

Our first study is on the singular values of the composite operator $\hat{A}$ as discrete approximation to the (compact) composition $A=B^{H}\circ J$. We choose $N=M=10000$ and use MATLAB to compute the first 20 singular values of the matrices, as this is sufficient to convey the message. Figure \ref{fig:svd_AJ} shows that the singular values of $\hat{A}$ decay exponentially fast, as in the half-logarithmic scale exponential decay appears linear. It is curious to note that $\sigma_i(\hat{A})\propto \sigma_i( \hat B^{H})\sigma_i(\hat J)$, since for general operators $A:X\rightarrow Z$, $B: Z\rightarrow Y$ one can only guarantee 
\begin{equation}\label{eq:sv_generals}
\sigma_{2i}(AB)\leq \sigma_i(A)\sigma_i(B),
\end{equation}
see \cite{Pie87}. 
It is also important to mention that in Figure \ref{fig:svd_AJ} we clearly have $\sigma_i(\hat J)\sim 1/i$, so for the compact operator $J$ theory and numerics agree. We conclude from this that numerically, the Hausdorff moment problem, even in composition with the integration operator, is exponentially ill-posed. We remark that this is true, in fact even more pronounced, when we use less than $10000$ discretization points or moments. We will come back to this in Section \ref{sec:explain} below.
\begin{figure}
\includegraphics[width=\linewidth]{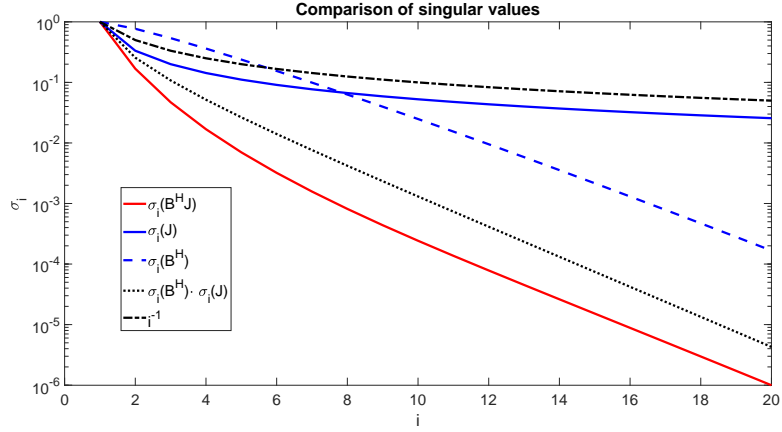}\caption{Singular values $\sigma_i$ of the discrete operators $\hat J$, $\hat{B}^{(H)}$, $\hat{A}$, and reference lines for exponential decay $\sigma_i=\exp(-i)$ and polynomial decay $\sigma_i=i^{-1}$.}\label{fig:svd_AJ}
\end{figure}

As a comparison, we now replace the Hausdorff moment operator with the multiplication operator which is known to preserve the ill-posedness of the integration operator. Figure \ref{fig:svd_MJ} shows that the singular values of the discretized composite operator $\hat A=B^{(M)}\circ J$ do not fall exponentially, but stay close to the rate $i^{-1}$ known from the theory. In this example we can calculate more singular values, and doing so we clearly see $\sigma_i(\hat A)\sim \sigma_i(\hat J)=1/i$. This result has been reported before in \cite{Freitag05}, but the contrast to the HM-operator is still surprising. Before we make an attempt at an explanation for this discrepancy in the next section, we proceed with another numerical experiment.

\begin{figure}
\includegraphics[width=\linewidth]{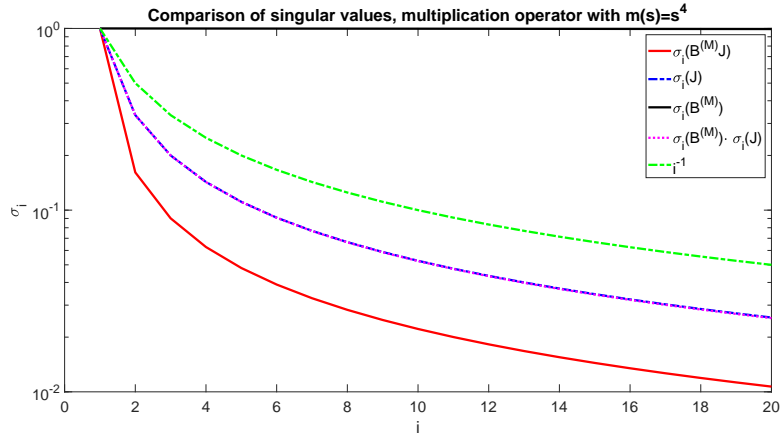}\caption{Singular values $\sigma_i$ of the discrete operators $\hat J$, $\hat{B}^{(M)}$ with $m(s)=s^4$, $\hat{A}$ as discretization of $A=B^{(M)}\circ J$, and reference line for polynomial decay $\sigma_i=i^{-1}$.}\label{fig:svd_MJ}
\end{figure}

The second indicator for the estimation of the decay of the singular values in \cite{HM} was the smoothness of the kernel $k(s,t)$ from \eqref{eq:kernel} and the singular values of the associated integral operator $[A^\ast Ax](s)=\int_0^1 k(s,t)x(t)\,dt$. We discretize this as special case of the integral $I_k[x](s)$ from \eqref{eq:int_prototype} with $M=N=5000$ supporting point for $s$ and $t$, respectively. As cut-off for the summation in the kernel we chose several values $j_{max}$ up to $j_{max}=200000$. The results are given in Figure \ref{fig:svd_AA}. We see that the singular values $\sigma_i$ increase slowly with increasing $j$, but even with $j_{max}$ they fall exponentially fast. For comparison we also plotted the singular values $\sigma_i(\hat A)^2$ where $\sigma_i(\hat A)$ are directly taken from Figure \ref{fig:svd_AJ}. From theory it is known that $\sigma_i(A)^2=\sigma_i(A^\ast A)$, and we indeed obtain a reasonable fit. We also plotted the theoretical value $\sigma(A^\ast A)\sim i^{-3}$ from \cite{HM}, which is clearly not met.

\begin{figure}
\includegraphics[width=\linewidth]{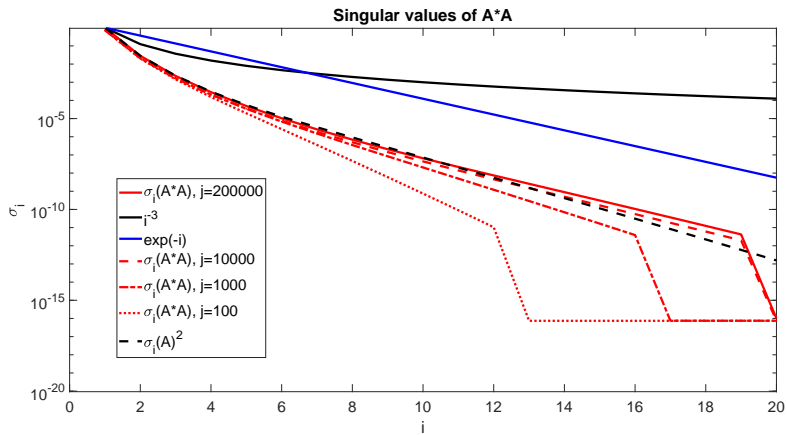}\caption{Computed singular values of $A^*A$ with kernel \eqref{eq:kernel} with truncation indices $j_{max}\in \{100,1000,10000,200000\}$. The singular values increase with $j_{max}$, although the difference between $j_{max}=10000$ and $j_{max}=200000$ is almost negligible. We can not observe the decay values $\sigma_i\sim i^{-3}$ we would expect. However, we have a good fit to the squared singular values $\sigma_i(A)^2$ previously computed and shown in Figure \ref{fig:svd_AJ}. Note that, numerically, $\mathrm{rank}(A^*A)$ is between $13$ and $20$ here.} \label{fig:svd_AA}
\end{figure}

As we have seen, all numerical experiments, even with a fairly large amount of discretization points, indicate an exponential decay of the singular values of the discretized composite operator $A=B^{(H)}\circ J$. Whether this is due to the singular values decaying exponentially fast also in the infinite setting, or due to numerical effects can still not be answered with certainty. Below we present an argument that suggests the latter and might reconcile theory and numerics.

\section{Explaining the seeming inconsistency}\label{sec:explain}
The focus in this section is on the singular values of the discrete approximations to the non-compact operators $B^H$ and $B^M$.

 For the Hausdorff-operator $B^H$ we recall from \cite{GHHK21} that $B^H$ can be seen as a Cholesky-factor of the Hilbert matrix, as it was shown that $B^H=L Q$, where $Q:L^2[0,1]\rightarrow \ell^2$ is an isometry and $L$ is triangular with $LL^T=H$ where $H$ is the Hilbert matrix,
\[
H_{ij}=\left(\frac{1}{i+j-1}\right)_{i,j=1}^\infty.
\]
For an explicit formula for the entries $L_{ij}$ we refer to \cite{GHHK21}.
The Hilbert matrix $H$ is, as infinite dimensional matrix mapping from $\ell^2$ to $\ell^2$, non compact with purely continuous spectrum $[0,\pi]$ \cite{Magnus50}. In the discrete setting, we truncate $H$ and $L$  to the first $n$ rows and donate these objects by $H_n$ and $L_n$. Note that these operators are integral-free, and $\sigma_i(H_n)=\sigma_i(L_n)^2$. This is also a main reason why believe the integration method used in the previous section is not crucial for the determination of the singular values, and that the behaviour of the numerical singular values is deeper rooted in the operator $B^H$ and its approximations $L_n$.

A big advantage of the identities above is that we can use results on the rather well-studied Hilbert matrix and its truncation, for which have good knowledge of the smallest and largest singular values, but lack details in between. More precisely, we have $\|H\|_{\ell^2\rightarrow \ell^2}=\pi$ \cite{Tal87}, and $H_n$ has, for sufficiently large $n$, eigenvalues arbitrarily close to $\pi$. This follows by setting $x_r=\frac{1}{\sqrt{r}}$, $r=1,\dots,N$ and $x_r=0$ for $r>N$, and observing $\langle Hx_r,x_r\rangle_{\ell^2\times\ell^2}/\|x_r\|^2\rightarrow \pi$ as $N\rightarrow \infty$ \cite[Theorem 323 and following pages]{polya}. As consequence, it follows $\|H_n\|\rightarrow \pi$ as $n\rightarrow \infty$, and $\sigma_1(H_n)\rightarrow \pi$. It is also well-known that the smallest eigenvalue of the truncated Hilbertmatrix $H_n$ is $\sigma_n(H_n)\sim \exp(-cn)$, see, e.g., \cite{Tal87}. However, a precise characterization of the intermediate singular values seems to be missing. A valuable hint however can be found in \cite{Beckermann}. There it was shown that 
\begin{equation}\label{eq:beckermann}
\sigma_{i+1}(H_n)\leq 4\left[ \exp\left(\frac{\pi^2}{2\log(8n-4)} \right)\right]^{-2i}\sigma_1(H_n), \qquad 1\leq i\leq n-1.
\end{equation}
Since we could not find a lower bound on $\sigma_{i+1}(H_n)$, the following arguments involve some speculation, but another comparison to the multiplication operator at the end of the section makes them appear reasonable. First, note that $\sigma_1(H_n)\leq\pi$ for all $n$, so this factor is not of interest.
Next, we remark that for fixed $n$, \eqref{eq:beckermann} yields exponential decay of the singular values of $H_n$. Because of \eqref{eq:sv_generals}, we have
\begin{equation}\label{eq:svs:bhj_discfrete}
\sigma_{2i}(\hat B^H \hat J)\leq \sigma_i(L_n)\sigma_i(\hat J)\leq 2\left[ \exp\left(\frac{\pi^2}{2\log(8n-4)} \right)\right]^{-i}\sqrt{\pi} \frac{1}{i}
\end{equation}
This shows that the singular values of the discrete approximation to $B^H\circ J$ fall exponentially and supports our claim that the results of our numerical experiment in Section \ref{sec:numerics} is correct and not due to incorrect approximation of the integrals.

The upper bound in \eqref{eq:beckermann} depends on the truncation index $n$. The term in brackets decreases monotonically with 
\[
\lim_{n\rightarrow \infty} \exp\left(\frac{\pi^2}{2\log(8n-4)} \right)=1.
\]
However, the convergence is excruciatingly slow, compare a plot in Figure \ref{fig:beckermann_expterm}. For $n=10^{30}$ the term has the value 1.072.  
\begin{figure}
\includegraphics[width=\linewidth]{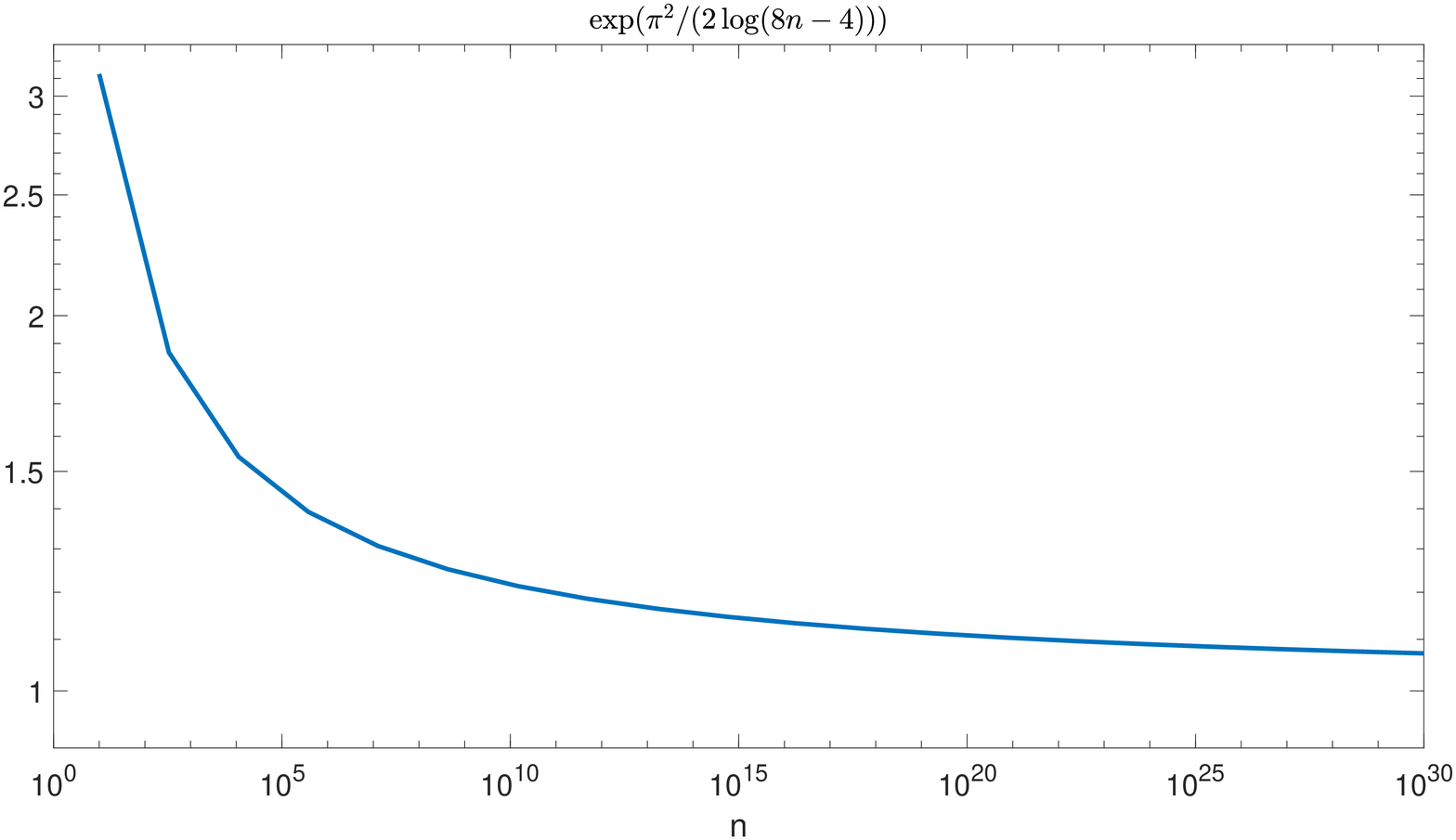}\caption{Plot of the $\exp$-term in \eqref{eq:beckermann} as function of $n$. The convergence is extremely slow.}\label{fig:beckermann_expterm}
\end{figure}

  In the limit $n\rightarrow \infty$, \eqref{eq:beckermann} therefore reduces to $\sigma_{i+1}(H_n)\leq 4\sigma_1(H_n)$, and because $\sigma_i(H_n)\leq \sigma_1(H_n)\leq \|H_n\|\leq \pi$, we even have $\sigma_i(H_n)\leq \pi$. We conjecture that indeed every singular value (with fixed index) convergerges slowly to $\pi$,
\begin{equation}\label{eq:conjecture}
\lim_{n\rightarrow \infty} \sigma_i(H_n)=\pi.
\end{equation}
Numerically, we see that for each index $i$ the sigular values $\sigma_i(H_n)$ as functions of $n$ are increasing, see Figure \ref{fig:svs_HK}. However, in terms of Figure \ref{fig:beckermann_expterm}, we can not increase the discretization enough, since no computer can hold such a matrix. Nevertheless, the increase of the singular values we can observe adds plausibility to \eqref{eq:conjecture}.
\begin{figure}
\includegraphics[width=\linewidth]{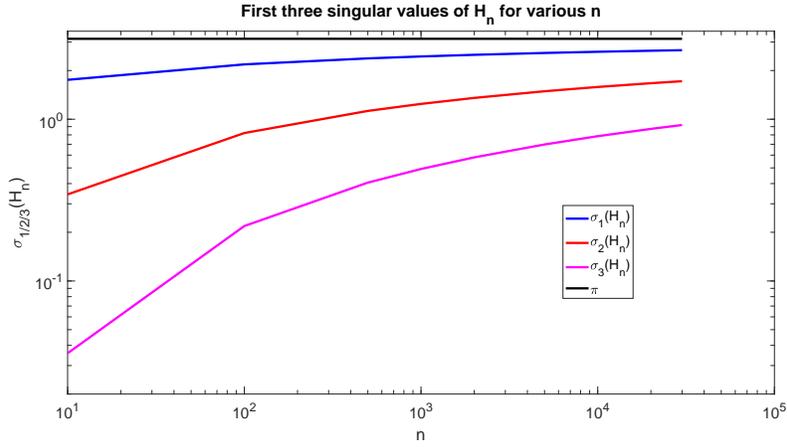}\caption{Singular values of the truncated Hilbert matrix various values of $n$. As $n$ increases, each singular value increases.}\label{fig:svs_HK}
\end{figure}

Below we will prove a result analogous to \eqref{eq:conjecture} for the multiplication operator. Before that, we discuss further consequences of \eqref{eq:beckermann}. In the limit $n\rightarrow \infty$ we obtain from \eqref{eq:svs:bhj_discfrete} the same upper bound for the discrete singular values as for the infinite dimensional ones, compare $\sigma_i(\hat B^H \hat J)\sim 1/i$ as limit of \eqref{eq:svs:bhj_discfrete} and the right-hand side of \eqref{eq:result_MH}. This means that the numerical experiments do not invalidate the theory of \cite{HM}, despite the huge gap. In fact, one can see numerically that, similarly to the ones of the truncated Hilbert matrix in Figure \ref{fig:svs_HK}, the singular values of $\hat A$ increase (again very slowly) when the discretization is refined, see Figure \ref{fig:svd_M}

\begin{figure}
\includegraphics[width=\linewidth]{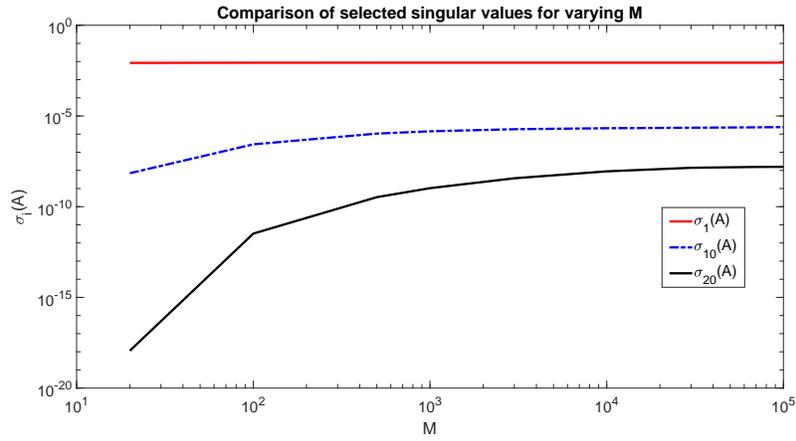}\caption{Singular values $\sigma_1(\hat A),\sigma_{10}(\hat A))$ and $\sigma_{20}(\hat A)$ as function of the maximal moment discretization $M$. It appears that the singular values converge with increasing $M$. In particular, numerically they do not increase to approximate the rate $\sigma_i\sim i^{-1.5}$.}\label{fig:svd_M}
\end{figure}

We close this section by checking analogies of the above results for the multiplication operator.  As mentioned before, the spectrum of $B^{M}$ consists of all values of the multiplicator function $m(s)$, $0\leq s\leq 1$. Due to the simple structure of the multiplication operator, we can prove that each singular value of its discrete approximation converges to a specific value, so the analogon of the conjecture \eqref{eq:conjecture} is true for $B^M$.
\begin{lemma}
Let $m(s)\in L^\infty[0,1]$ be the multiplicator function of the operator $B^{(M)}$ from \eqref{eq:mult} with $m_{max}:=\mathrm{ess\,sup}_{s\in[0,1]} m(s)<\infty$. Let $s_k=\frac{k-1}{K-1}$, $k=1,\dots,K$, $K\in \N$ be discretizations of the interval $[0,1]$ in $K$ points. Then the diagonal matrices $M_K$ with entries $M_K(k,k)=\tilde m(s_k)$, where $\tilde m(s_k)$ is the decreasing rearrangement of $m(s_k)$, and zero outside the diagonal are a discrete approximation to $B^{(M)}$ and 
\[
\lim_{K\rightarrow \infty} \sigma_i(M_K)= m_{max}.
\]
for each $i\in N$.
\end{lemma}
\begin{proof}
Because $M_K$ is diagonal its singular values are the decreasing rearrangement of the diagonal entries. W.l.o.g. we consider $m(s)$ to be decreasing further. Because $s_k-s_{k-1}\rightarrow 0$ as $K$ increases, there is $\bar K=K(i,\epsilon)$ such that for each fixed $i$ and arbitrary but fixed $\epsilon>0$ $\sigma_i(M_K)=m(s_i)>m_{max}-\epsilon$ for $K>\bar K$.
\end{proof}

Numerically this can be confirmed easily. Since the simgular values of the multiplication operator are trivial, we can plot them all. We have done so in Figure \ref{fig:svs_MK} for varying $K$. For each fixed index $i$ we see $\sigma_i(M_K)\rightarrow 1$ as $K$ increases. The relation \eqref{eq:sv_generals} hence yields the bound $\sigma_{2i}(\hat B^M \circ \hat J)\leq \sigma_i(\hat B^M)\sigma_i(\hat J)\sim i$ when $B^M$ is discretised fine enough.
\begin{figure}
\includegraphics[width=\linewidth]{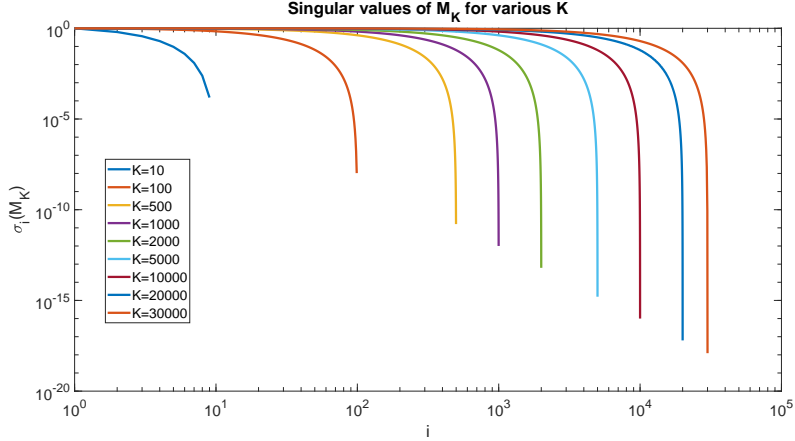}\caption{Singular values of the discrete multiplication operator with $m(s)=s^4$ for various values of $K$. As $K$ increases, each singular value converges to $1$.}\label{fig:svs_MK}
\end{figure}

To summarize this section and conclude the paper, we have demonstrated the possibility that the numerically observed exponentially decaying singular values of the composition $A=B^H\circ J$ can be explained through the inadequate numerical convergence of the spectrum of the discrete approximations in terms of discretization level. Our results suggest that a discretization level that is impossible to realize one might see singular values dropping only (approximately) polynomially fast. In this sense, the results of \cite{HM} do not contradict our experiments in Section \ref{sec:numerics}.

\section*{Acknowledgment}
Daniel Gerth has been supported in parts by the German Science Foundation (DFG) under the grant GE~3171/1-1 (Project No.~416552794). The author thanks Bernd Hofmann (TU Chemnitz) for the discussions and encouraging the writing of this paper.


\begin{thebibliography}{99}


\bibitem{Beckermann} {B.~Beckermann and A.~Townsend},  {\it Bounds on the singular values of matrices with displacement structure}, SIAM Review, 61(2):319--344, 2019.




\bibitem{EHN}  {H.~W.~Engl, M.~Hanke and A.~Neubauer}, {\it
    Regularization of Inverse Problems}, Kluwer, Dordrecht, 1996.

\bibitem{Freitag05} {M.~Freitag and B.~Hofmann}, {\it Analytical and numerical studies on the influence of multiplication operators for the ill-posedness of inverse problems}, J.~Inverse Ill-Posed Probl., 13(2):123--148, 2005.

\bibitem{GHHK21} {D.~Gerth, B.~Hofmann, C.~Hofmann and S.~Kindermann}, {\it The Hausdorff moment problem in the light of ill-posedness of type~I}, Eurasian Journal of Mathematical and Computer Applications,
9(2):57--87, 2021.

\bibitem{Hausdorff23} {F.~Hausdorff}, {\it Momentprobleme f\"ur ein endliches Intervall} (German), Math.~Z.,
16(1):220--248, 1923.

\bibitem{HM}  {B.~Hofmann and P.~Math\'{e}},
{\it The degree of ill-posedness of composite linear ill-posed problems with focus on the impact of the non-compact Hausdorff moment operator}
Preprint, arXiv:2111.01036v2, 2021.


\bibitem{HW05}  {B.~Hofmann and L.~von Wolfersdorf},  {\it Some results and a conjecture on the degree of ill-posedness for integration operators with weights}, Inverse Problems, 21(2):427--433, 2005.

\bibitem{HW09}  {B.~Hofmann and L.~von Wolfersdorf}, {\it A new result on the singular value asymptotics of integration oprators with weights}, J.~Integral Equations Appl., 21(2):281--295, 2009.



\bibitem{Magnus50}   {W.~Magnus}, {\it On the spectrum of Hilbert's matrix}, Amer.~J.~ Math., 72:699--704, 1950.


\bibitem{Nashed} {M.~Z.~Nashed}, {\it A new approach to classification and regularization of ill-posed operator equations},
In: Inverse and Ill-posed Problems Sankt Wolfgang, 1986, volume~4 of Notes Rep.~Math.~Sci.~Engrg. (Eds.: H.~W.~Engl and C.~W.~Groetsch),
Academic Press, Boston, 1987, pp.~53--75.



\bibitem{Pie87} {A.~Pietsch}, {\it Eigenvalues and $s$-numbers},
  Cambridge Studies in Advanced Mathematics, Cambridge University
  Press, Cambridge, 1987.



\bibitem{polya} { G.~H.~Hardy, J.~E.~Littlewood and G. P{\'o}lya}, \emph{Inequalities} Cambridge University Press, Cambridge, 1952




\bibitem{Tal87}{G.~Talenti}, {\it Recovering a function from a finite number of moments}, Inverse Problems, 3(3):501--517, 1987.

\bibitem{Todd54} {J.~Todd}, {\it The condition number of the finite segment of the Hilbert matrix},
Nat.~Bur.~of Standards Appl.~Math.~Series, 39:109--116, 1954.



\end{thebibliography}
\end{document}